\documentclass[12pt]{article}
\usepackage{hyperref}
\usepackage[active]{srcltx}
\usepackage[curve]{xypic}
\usepackage{graphicx}
\usepackage{longtable}
\usepackage{rotating}
\usepackage{multirow}
\usepackage[active]{srcltx}
\usepackage{epsfig,amsmath}
\usepackage[english]{babel}
\usepackage{amsmath,amsfonts,amssymb,amscd,color,epsfig,amsthm}
\usepackage{titletoc}
\usepackage{mathrsfs}
\usepackage{appendix}
\usepackage{enumitem}
\setenumerate[1]{itemsep=0pt,partopsep=0pt,parsep=\parskip,topsep=5pt}
\setitemize[1]{itemsep=0pt,partopsep=0pt,parsep=\parskip,topsep=5pt}
\setdescription{itemsep=0pt,partopsep=0pt,parsep=\parskip,topsep=5pt}
\usepackage{caption}

\newtheorem{theorem}{Theorem}[section]
\newtheorem{lemma}[theorem]{Lemma}

\def\g{\gamma}

\def\la{\lambda}

\def\th{\theta}
\def\wt{\widetilde}

\def\C{\mbox{$\mathbb C$}}

\def\D{\mbox{$\mathbb D$}}

\def\MMM{{\mathcal M}}

\def\HH{{\mathbf H}}
\def\MM{{\mathbf M}}
\def\BB{{\mathbf B}}

\def\cbar{\overline{\C}}
\def\smm{{\backslash}}

\def\ds{\displaystyle}

\begin{document}
\title{On the boundary of the central quadratic \\ hyperbolic component
}
\author{Guizhen Cui\thanks{The first author is supported by National Key R\&D Program of China No. 2021YFA1003203, and the NSFC Grant Nos. 12131016 and 12071303.}
\and Wenjuan Peng\thanks{The second author is supported by the NSFC Grant Nos. 12122117, 12271115 and 12288201.}}
\date{\today}
\maketitle
\begin{abstract}
We give a concrete description for the boundary of the central quadratic hyperbolic component. The connectedness of the Julia sets of the boundary maps are also considered.
\end{abstract}

\section{Introduction}
Let $\MMM_d$ denote the complex orbifold of holomorphic conjugate classes of rational maps of degree $d\ge 2$. A rational map is {\bf hyperbolic} if the orbit of every critical point converges to attracting (or super-attracting) periodic cycles under iteration. Hyperbolic maps form an open subset of $\MMM_d$ and conjecturally they are dense in $\MMM_d$ \cite{MSS}. A connected component of this open subset is called a {\bf hyperbolic component}. Any two maps in the same hyperbolic component are quasiconformally conjugate in a neighborhood of their Julia sets \cite{MSS}.

The hyperbolic components of quadratic rational maps have been intensively studied and fruitful results have been obtained. For example, Rees \cite{R} gave a classification of the hyperbolic components (see also \cite{Mil1}), and proved the boundedness of certain one real dimension loci in hyperbolic components. Epstein \cite{E} provided a boundedness result for a hyperbolic component of quadratic rational maps possessing two distinct attracting cycles. The closure of the set of quadratic rational maps having a degenerate parabolic fixed point was investigated by Buff-\'{E}calle-Epstein in \cite{BEE}.

In this work, we consider the boundary of the hyperbolic component $\HH$ in $\MMM_2$ containing $z\mapsto z^2$, which is called the {\bf central} quadratic hyperbolic component. Every rational map in $\HH$ has two fixed attracting (or super-attracting) points and its Julia set is a quasicircle.

Each quadratic rational map has exactly three fixed points (counting the multiplicity). Let $\la_1,\la_2,\la_3$ be the eigenvalues at fixed points. These three eigenvalues determine the quadratic rational map up to holomorphic conjugacy (\cite[Lemma 3.1]{Mil1}). If the three fixed points are distinct, then by the Rational Fixed Point Theorem (\cite[Theorem 12.4]{Mil2}), we have
\begin{eqnarray}
\frac{1}{1-\la_1}+\frac{1}{1-\la_2}+\frac{1}{1-\la_3}=1.
\end{eqnarray}

If a quadratic rational map has at least two distinct fixed points, conjugating the map by a M\"{o}bius transformation if necessary, we may assume that $0$ and $\infty$ are the two fixed points and the eigenvalues at them are $\la_1$ and $\la_2$ respectively. Thus the holomorphic conjugacy class of the map can be represented by
$$
f_{\la_1,\la_2}(z)=\frac{\la_1z+z^2}{\la_2z+1}
$$
with $\lambda_1\lambda_2\neq 1$. Note that $f_{\la_1,\la_2}$ is holomorphically conjugated to $f_{\la_2,\la_1}$. Therefore the central quadratic hyperbolic component can be expressed by
$$
\HH=\{[f_{\la_1,\la_2}]:\, (\la_1,\la_2)\in\D\times\D\},
$$
where $\D$ denotes the unit disk.

Recall that
\begin{eqnarray*}
\mathrm{Per}_1(1)&=&\{[f]\in\MMM_2:\, f\ \text{has a fixed point with eigenvalue 1}\}\text{ and} \\
\mathbf{M}_1&=&\{[f]\in\mathrm{Per}_1(1):\, J(f)\ \text{is connected}\},
\end{eqnarray*}
where $J(f)$ denotes the Julia set of $f$. One may refer to \cite{BE,Mil1,PR} for more properties about $\mathbf{M}_1$, see Figure 1 for a picture of $\MM_1$ parameterized by the eigenvalue of the other fixed point.

\begin{figure}[http]
\centering
\includegraphics[width=12cm]{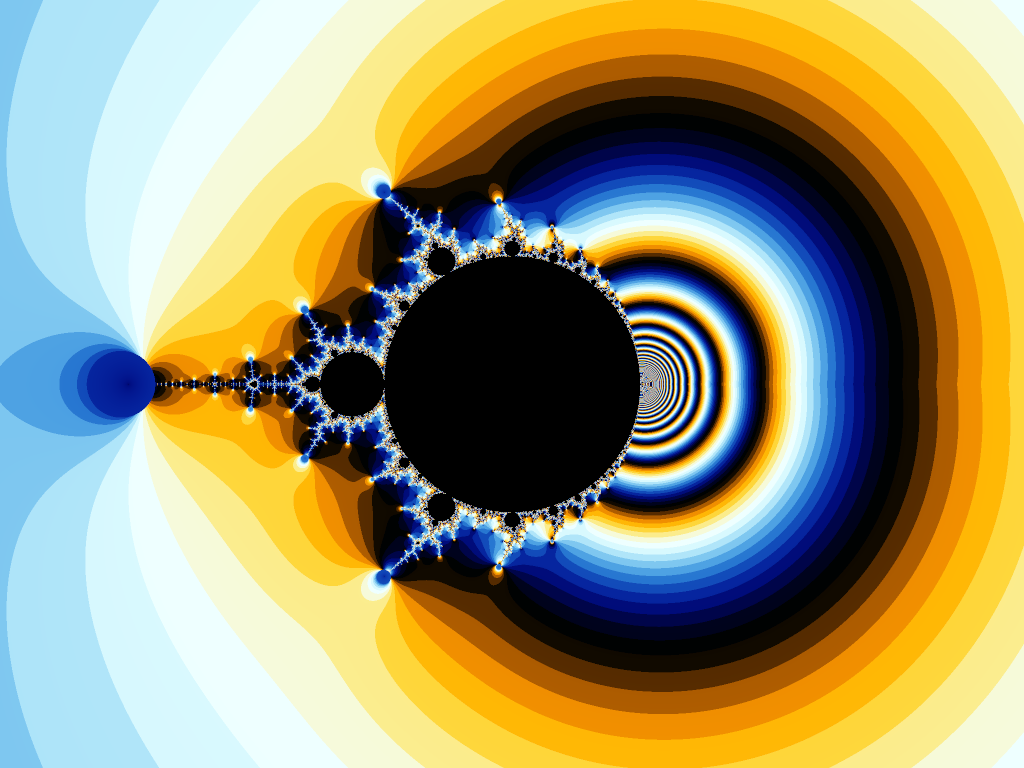}
\caption{A picture of $\MM_1$}
\end{figure}

Denote
\begin{eqnarray*}
\BB_0&=&\{[f_{\la_1,\la_2}]:\, (\la_1,\la_2)\in\overline{\D}\times\partial\D,\text{ but } \la_1\ne1, \la_2\ne1 \text{ and }  \la_1\la_2\ne1\},\\
\BB_1&=&\{[f_{1,\la}]:\, \la\in \overline{\D}\smm\{1\}\}, \\
\BB_2&=&\{[f_{1,\la}]:\, \mathrm{Re}\,\la>1\}.
\end{eqnarray*}

From the definitions, we easily make the following deductions.
\begin{itemize}
\item[(a)] Obviously, $\BB_0$, $\BB_1$ and $\BB_2$ are pairwise disjoint.

For any $[g]\in\BB_0\cup \BB_1$, either $g$ has three distinct fixed points or $g$ has no repelling fixed points. Note that $$
\overline{\BB}_2=\{[f_{1,\la}]:\, \mathrm{Re}\,\la\ge 1\text{ but }\la\neq 1\}\cup\{[R]\},
$$
where $R(z)=z+1/z$. So for any $[g]\in\overline{\BB}_2$, either $g$ has exactly two fixed points and one of them is repelling, or $g$ has a unique fixed point. Thus $(\BB_0\cup \BB_1)\cap\overline{\BB}_2=\emptyset$.

\item[(b)] It is clear that $\BB_0\cap \mathrm{Per}_1(1)=\emptyset$ and $\BB_1\cup \overline{\BB}_2\subset \mathrm{Per}_1(1)$.

\item[(c)] By \cite[Lemma 8.2]{Mil1}, for any $[f]\in\MMM_2$, either $J(f)$ is connected or $J(f)$ is a Cantor set. In the latter case, both critical orbits of $f$ converge to an attracting or a parabolic fixed point. If $[g]\in\BB_0\cup\BB_1$, then $g$ has two distinct non-repelling fixed points. Thus $J(g)$ is connected. Actually,  $\BB_1\cup \{[R]\}$ is the closure of the central component of the interior of $\MM_1$.
\end{itemize}

 Our main result is the following.

\begin{theorem}\label{main}
The following statements hold:
\begin{itemize}
\item[(1)] $\partial\HH=\BB_0\cup\BB_1\cup\overline{\BB}_2$;
\item[(2)] $\BB_2\cap \mathbf{M}_1=\emptyset$;
\item[(3)] $\partial\HH\cap(\mathrm{Per}_1(1)\smm\mathbf{M}_1)\ne\mathrm{Per}_1(1)\smm\mathbf{M}_1$.
\end{itemize}
\end{theorem}

\noindent{\bf Remark.} 1. We highly appreciate that the referee makes us aware of Buff-Epstein's example which illustrates that
$$
\MM_1 \subset \{\la\in\C:\, |\la+1|\le 2\},
$$
where $\MM_1$ is parameterized by the eigenvalue of the other fixed point. Refer to \cite[pp. 272: Example]{BE}. Theorem \ref{main} (2) could be obtained directly by their illustration, and moreover, one could conclude that
$$
\overline{\BB}_2\cap \MM_1=\{[R]\}
$$
from their example.

In Section 3, we exploit twist surgery to prove $$
\BB_2\subset\partial\HH\cap(\mathrm{Per}_1(1)\smm\mathbf{M}_1),
$$
which is the main part of this work.

To prove Theorem \ref{main} (3), we discuss the boundedness of $\MM_1$ in Section 4 and give a bound on $\MM_1$ in Lemma \ref{bound}. The bound given by Buff-Epstein is stronger than Lemma \ref{bound}. We would like to point out that, compared to their result, Lemma \ref{bound} is obtained by quite simple and elementary computations. For self-containedness, we include Lemma \ref{bound} in this article.

Obviously, $\partial\HH\cap \mathrm{Per}_1(1)=\BB_1\cup\overline{\BB}_2$ by Theorem \ref{main} (1).

\vspace{2mm}

2. One typical way to produce boundary maps of hyperbolic components is pinching. It is proved in \cite{CT} that under certain conditions, the pinching path $\{f_t\}$ starting from a geometrically finite rational map $f$ converges uniformly to a geometrically finite rational map $g$, and the quasiconformal conjugacy from $f$ to $f_t$ converges uniformly to a semi-conjugacy from $f$ to $g$. Moreover, $M[g,J]\subset\partial M[f,J]$, where $M[f,J]\subset\MMM_d$ is defined as $[h]\in M[f,J]$ if $h$ is quasiconformally conjugate to $f$ in a neighborhood of the Julia sets.

It turns out that any geometrically finite rational map $g$ with $[g]\in\BB_0\cup\BB_1$ is the limit of a pinching path starting from a rational map in $\HH$. By Theorem \ref{main} (2), if $[g]\in\BB_2$, then $J(g)$ is disconnected. Thus it can not be obtained by pinching from $\HH$.

It is obvious that $\HH=M[f,J]$ for any $[f]\in\HH$; and $\mathrm{Per}_1(1)\smm\mathbf{M}_1=M[g,J]$ for any $[g]\in\mathrm{Per}_1(1)\smm\mathbf{M}_1$ (\cite{MSS}). By Theorem \ref{main} (3), for $[g]\in\BB_2$, $[g]\in\partial M[f,J]$ but $M[g,J]$ is not contained in $\partial M[f,J]$.

\section{A direct description of $\partial\HH$}
In this section, we give a direct description of the boundary $\partial\HH$ by a basic discussion.

\begin{lemma}\label{repelling}
For any $[f_{\lambda_1,\lambda_2}]\in\HH$, let $\lambda_3$ be the eigenvalue of $f_{\lambda_1,\lambda_2}$ at the unique repelling fixed point. Then $\mathrm{Re}\,\la_3>1$.
\end{lemma}

\begin{proof}
Let $\g(z)=1/(1-z)$. Then $\g$ maps $\D$ onto $\{w\in\C:\,\mathrm{Re}\,w>1/2\}$. Combining with the equation (1), we have
$$
\mathrm{Re}\left(1-\frac{1}{1-\la_3}\right)=\mathrm{Re}\left(\frac{1}{1-\la_1}+\frac{1}{1-\la_2}\right)>1.
$$
So $\mathrm{Re}~(1/(1-\la_3))<0$ and hence $\text{Re}\,\la_3>1$.
\end{proof}

\begin{lemma}\label{eigenvalue}
Let $f_n$ be quadratic rational maps with eigenvalues $\la_1(f_n),\la_2(f_n),\la_3(f_n)$ at the fixed points. Then $\{[f_n]\}$ converges to $[g]$ in $\MMM_2$ if and only if by a suitable ordering, $((\la_1(f_n), \la_2(f_n), \la_3(f_n))$ converges to $(\la_1, \la_2, \la_3)$ in $\C^3$, where $\la_1, \la_2, \la_3$ are the eigenvalues of $g$ at the fixed points.
\end{lemma}

\begin{proof}
Assume that $[f_n]$ converges to $[g]$ in $\MMM_2$. Since $g$ always has a fixed point with non-zero eigenvalue, we may represent $[g]$ by
$$
g(z)=\la_1 z+B+1/z,
$$
with $\la_1\ne 0$.
Note that $g$ is holomorphically conjugate to $z\mapsto\la_1 z-B+1/z$.

As $n>0$ is large enough, $\la_1(f_n)\neq 0$. So $[f_n]$ can also be represented by
$$
f_n(z)=\la_1(f_n) z+B_n+1/z.
$$
Consequently, $\{B_n^2\}$ converges to $B^2$ as $n\to\infty$.

The eigenvalues of the other two fixed points of $f_n$ are
\begin{equation}
1-\frac{B_n^2}{2}\pm B_n\sqrt{\frac{B_n^2}{4}-[\la_1(f_n)-1]}.
\end{equation}
Thus $(\la_2(f_n),\la_3(f_n))$ converges to $(\la_2,\la_3)$ since $\{B_n^2\}$ converges to $B^2$.

Conversely, assume $((\la_1(f_n), \la_2(f_n), \la_3(f_n))$ converges to $(\la_1, \la_2, \la_3)$ in $\C^3$. Then at least one of them, say $\la_1$ is non-zero. Thus as $n>0$ is large enough, $\la_1(f_n)\neq 0$. As above, $[f_n]$ can also be represented by $f_n(z)=\la_1(f_n) z+B_n+1/z$. By equation (2), we have
$$
\la_2(f_n)+\la_3(f_n)=2-B_n^2.
$$
From the condition $((\la_1(f_n), \la_2(f_n), \la_3(f_n))$ converges to $(\la_1, \la_2, \la_3)$ in $\C^3$, we obtain $\{B_n^2\}$ converges to $B^2$. This yields that $\{[f_n]\}$ is convergent in $\MMM_2$.
\end{proof}

The reader may also refer to \cite[Section 3]{Mil1} or \cite[Lemma 3]{E} for the above lemma. The statement here is slightly different from them.

\begin{lemma}\label{12}
The following statements hold:
\begin{itemize}
\item[(i)] $\partial\HH\subset\BB_0\cup\BB_1\cup\overline{\BB}_2$;
\item[(ii)] $\BB_0\cup\BB_1\subset\partial\HH$.
\end{itemize}
\end{lemma}

\begin{proof}
(i) Assume that $[g]\in\partial\HH$. Then there is a sequence $\{[f_n]\}$ in $\HH$ which converges to $[g]$. Let $\la_j(f_n)$ ($j=1,2,3$) be the eigenvalues of $f_n$ at the three fixed points with $|\lambda_3(f_n)|>1$. Then $(\la_1(f_n), \la_2(f_n), \la_3(f_n))$ converges to $(\la_1, \la_2, \la_3)$ in $\C^3$ by Lemma \ref{eigenvalue}.

If $\la_1\ne 1$ or $\la_2\ne 1$, then the two attracting fixed points of $\{f_n\}$ converge to two distinct fixed points of $g$. Moreover $\la_1\la_2\ne 1$. Thus $[g]=[f_{\lambda_1,\lambda_2}]\in\BB_0\cup\BB_1$. If $\la_1=\la_2=1$, then $\mathrm{Re}\,\la_3\ge 1$ by Lemma \ref{repelling}. Thus $[g]\in\overline{\BB}_2$. Now we have proved $\partial\HH\subset\BB_0\cup\BB_1\cup\overline{\BB}_2$.

(ii) For any $[f_{\lambda_1,\lambda_2}]\in\BB_0\cup\BB_1$, one can choose a sequence in $\HH$ which converges to $[f_{\lambda_1,\lambda_2}]$. Thus $\BB_0\cup\BB_1\subset\partial\HH$.
\end{proof}

\section{Dynamics of maps in $\BB_2$}
In this section, we will apply twist deformation to study the dynamics of maps in $\BB_2$.

Let $[f_{\la_1,\la_2}]\in\HH$ with $\la_1,\la_2\ne 0$. Denote by $U$ the set of points in the Fatou set $F(f_{\la_1,\la_2})$ with infinite forward orbits. Then the quotient space $U/\sim$ is a disjoint union of two tori $T_1$ and $T_2$ under the equivalence relation $z_1\sim z_2$ if $f^n(z_1)=f^m(z_2)$ for some integers $n,m\ge 0$. Denote by $\pi: U\to U/\sim$ the natural projection. Then each torus $T_j$ contains a unique point $x_j$ such that $\pi(c_j)=x_j$, where $c_1$ and $c_2$ are the critical points of $f_{\la_1,\la_2}$.

By the Koenigs Linearization Theorem (\cite{Mil2}), for $j=1,2$, there is a conformal map $\iota_j:\, \C/\Lambda_j\to T_j$, where $\Lambda_j$ is the lattice generated by $2\pi i$ and $\omega_j$ with $e^{\omega_j}=\la_j$, such that for any simple closed curve $\beta\subset\C/\Lambda_j$ corresponding to $\zeta\mapsto\zeta+2\pi i$, each component of $\pi^{-1}(\iota_j(\beta))$ is a Jordan curve in $U$ if $\iota_j(\beta)\subset T_j$ is disjoint from the point $x_j$.

For $j=1,2$, let $\g_j\subset T_j\smm\{x_j\}$ be a simple closed curve such that $\iota_j^{-1}(\g_j)$ is homotopic to the curve corresponding to $\zeta\mapsto\zeta+\omega_j$ in $\C/\Lambda_j$. Then each component of $\pi^{-1}(\g_j)$ is an open arc. In particular, $\pi^{-1}(\g_j)$ has a unique component whose endpoints contain the attracting fixed points.

Let $\tau_j$ be the Dehn twist along $\g_j$. We will consider the quasiconformal deformation of $f_{\la_1,\la_2}$ whose projection to $T_1\cup T_2$ realizes the repeated Dehn twist $\tau_1^{-n}\circ\tau_2^{n}$. Such quasiconformal deformation can be defined as follows.

Take an annulus $A_j\subset T_j\smm\{x_j\}$ for $j=1,2$, such that $\partial A_j$ consists of two disjoint simple closed curves homotopic to $\g_j$ in $T_j\smm\{x_j\}$. Let $\chi_j$ be a conformal map from $A_j$ onto $\{z:\, 1<|z|< r_j\}$. Define a quasiconformal map $\Phi_j$ from $\{z:\, 1<|z|<r_j\}$ onto itself by
$$
\Phi_j(re^{i\theta})=r\exp\left[i\left(\th+(-1)^{j}2\pi\frac{\log r}{\log r_j}\right)\right].
$$
Now we define
$$
\Phi=
\begin{cases}
\chi_j^{-1}\Phi_j\chi_j\, & \text{ on } A_j, \\
\text{id}\, & \text{ otherwise}.
\end{cases}
$$

For every $n\ge 1$, let $\mu_n$ be the Beltrami differential of $\Phi^{\circ n}$. Let $\wt\mu_n$ be the pullback of $\mu_n$ under $\pi$, i.e.,
$$
\wt\mu_n(z)=
\begin{cases}
\mu_n(\pi(z))\frac{\overline{\pi'(z)}}{\pi'(z)}\, & \text{ for } z\in\pi^{-1}(A_1\cup A_2),\\
\text{0}\, & \text{ otherwise}.
\end{cases}
$$
Then there exists a quasiconformal map $\phi_n:\cbar\to\cbar$ with Beltrami differential $\wt\mu_n$. Set $f_n=\phi_n\circ f_{\la_1,\la_2}\circ\phi_n^{-1}$. Then $f_n$ is a rational map. We call $\{[f_n]\}$ a {\bf twist sequence} of $f_{\la_1,\la_2}$ along $(\omega_1,\omega_2)$.

\begin{theorem}\label{convergence}
The twist sequence $\{[f_n]\}$ converges to $[f_{1,\la}]$ as $n\to\infty$, where $\la$ satisfies
\begin{eqnarray}
\frac{1}{1-\la}=\frac{1}{\omega_1}+\frac{1}{\omega_2},
\end{eqnarray}
and $[f_{1,\la}]\in\BB_2$. Conversely, any $[f]\in\BB_2$ is the limit of a twist sequence in $\HH$.
\end{theorem}

\begin{proof}
Denote by $\lambda_j(f_n)$ ($j=1,2$) the eigenvalues of $f_n$ at the two attracting fixed points. The quotient space of $f_n$ is the disjoint union of two tori $T_{1,n}$ and $T_{2,n}$. By the Koenigs Linearization Theorem, for $j=1,2$, $T_{j,n}$ is holomorphically isomorphic to $\C/\Lambda_{j,n}$, where $\Lambda_{j,n}$ is the lattice generated by $2\pi i$ and $\omega_{j,n}$ with $e^{\omega_{j,n}}=\la_j(f_n)$. Since $f_n$ is the quasiconformal deformation of $f_{\la_1,\la_2}$ whose projection to $T_1\cup T_2$ realizes the repeated Dehn twist $\tau_1^{-n}\circ\tau_2^{n}$, we have $T_{j,n}$ is also holomorphically isomorphic to $\C/\tilde\Lambda_{j,n}$, where $\tilde\Lambda_{j,n}$ is the lattice generated by $2\pi i+(-1)^jn\omega_{j}$ and $\omega_{j}$, for $j=1,2$. Thus
$$
\frac{2\pi i+(-1)^jn\omega_{j}}{2\pi i}=\frac{\omega_{j}}{\omega_{j,n}}.
$$
Equivalently, we have
$$
\frac{1}{\omega_{j,n}}=\frac{1}{\omega_j}+\frac{(-1)^jn}{2\pi i}.
$$
Thus
$$
\frac{1}{\omega_{1,n}}+\frac{1}{\omega_{2,n}}=\frac{1}{\omega_1}+\frac{1}{\omega_2},
$$
and $\omega_{j,n}\to 0$ as $n\to\infty$. Note that
$$
\lim_{\omega\to 0}\left(\frac{1}{\omega}+\frac{1}{1-e^{\omega}}\right)=\frac{1}{2}.
$$
Thus
\begin{eqnarray*}
\frac{1}{\omega_1}+\frac{1}{\omega_2}&=&\lim_{n\to\infty}\left(\frac{1}{\omega_{1,n}}+\frac{1}{\omega_{2,n}}\right) \\
&=&1-\lim_{n\to\infty}\left(\frac{1}{1-e^{\omega_{1,n}}}+\frac{1}{1-e^{\omega_{2,n}}}\right) \\
&=&\lim_{n\to\infty}\frac{1}{1-\la_{3}(f_n)},
\end{eqnarray*}
where $\la_3(f_n)$ denotes the eigenvalue of $f_n$ at the repelling fixed point. Let $\la=\ds\lim_{n\to\infty}\la_3(f_n)$. Then the equation (3) holds.

Since $(\la_1(f_n),\la_2(f_n),\la_3(f_n))\to (1,1,\la)$, $\{[f_n]\}$ converges to $[f_{1,\la}]$. Note that $\la_1,\la_2\in\D$. Thus $\mathrm{Re}\,\omega_j<0$ for $j=1,2$. By the equation (3), we have $\mathrm{Re}\, 1/(1-\la)<0$. So $\mathrm{Re}\,\la>1$ and hence $[f_{1,\la}]\in\BB_2$.

Conversely, notice that the map
$$
\lambda(\omega_1,\omega_2)=1-\frac{\omega_1\omega_2}{\omega_1+\omega_2}
$$
from $\{\omega_1:\,\mathrm{Re}\,\omega_1<0\}\times\{\omega_2:\,\mathrm{Re}\,\omega_2<0\}$ to $\{\lambda:\,\mathrm{Re}\,\lambda>1\}$ is surjective. Thus for any $\la\in\C$ with $\mathrm{Re}\,\la>1$, there exist $\omega_1,\omega_2\in\C$ with $\mathrm{Re}\,\omega_1, \mathrm{Re}\,\omega_2<0$, such that
$$
\frac{1}{1-\la}=\frac{1}{\omega_1}+\frac{1}{\omega_2}.
$$
Choose $\la_1=e^{\omega_1}$ and $\la_2=e^{\omega_2}$. Let $\{[f_n]\}$ be the twist sequence of $f_{\la_1,\la_2}$ along $(\omega_1,\omega_2)$. Then $\{[f_n]\}$ converges to $[f_{1,\la}]$ by previous argument.
\end{proof}

\begin{lemma}\label{dynamics}
For any $[f_{1,\la}]\in\BB_2$, $J(f_{1,\la})$ is a Cantor set.
\end{lemma}

\begin{proof}
We will exploit the process of the twist deformation as above to study the dynamics of $[f_{1,\la}]$. Denote the two invariant attracting Fatou domains of $f_{\la_1,\la_2}$ by $D_1$ and $D_2$. Denote $\Omega_j=D_j\smm\pi^{-1}(\overline{A}_j)$. Then both $\Omega_1$ and $\Omega_2$ are domains containing a critical orbit by the choice of $A_j$. The quasiconformal map $\phi_n$ is conformal in $\Omega_1\cup\Omega_2$ for $n\ge 1$. We normalize $\phi_n$ by fixing the repelling fixed point of $f_{\la_1,\la_2}$ and two points in the backward orbit of it such that $f_n=\phi_n\circ f_{\la_1,\la_2}\circ \phi_n^{-1}$ converges uniformly to $f_{1,\la}$. Since such three points are not contained in $\Omega_1\cup\Omega_2$, the sequence $\{\phi_n\}$ is a normal family in $\Omega_1\cup\Omega_2$. Thus there exists a subsequence $\{\phi_{n_k}\}$ locally uniformly convergent to a map $\varphi$ defined on $\Omega_1\cup\Omega_2$ which is either conformal or a constant. If $\varphi$ is a constant in some $\Omega_j$, then one of the critical point of $f_{1,\la}$ is a fixed point. This is impossible. Thus $\varphi$ is conformal in $\Omega_1\cup\Omega_2$ and $f_{1,\la}(\varphi(\Omega_j))=\varphi(\Omega_j)$. Therefore $\varphi(\Omega_j)$ is contained in an invariant Fatou domain of $f_{1,\la}$. So each of the two critical points of $f_{1,\la}$ lie in an invariant Fatou domain of $f_{1,\la}$. Note that $f_{1,\la}$ has no attracting fixed points. So both of the critical points lie in the fixed parabolic Fatou domain. Thus $J(f_{1,\la})$ is a Cantor set \cite[Lemma 8.1]{Mil1}.
\end{proof}

\begin{figure}[http]
\centering
\includegraphics[width=12cm]{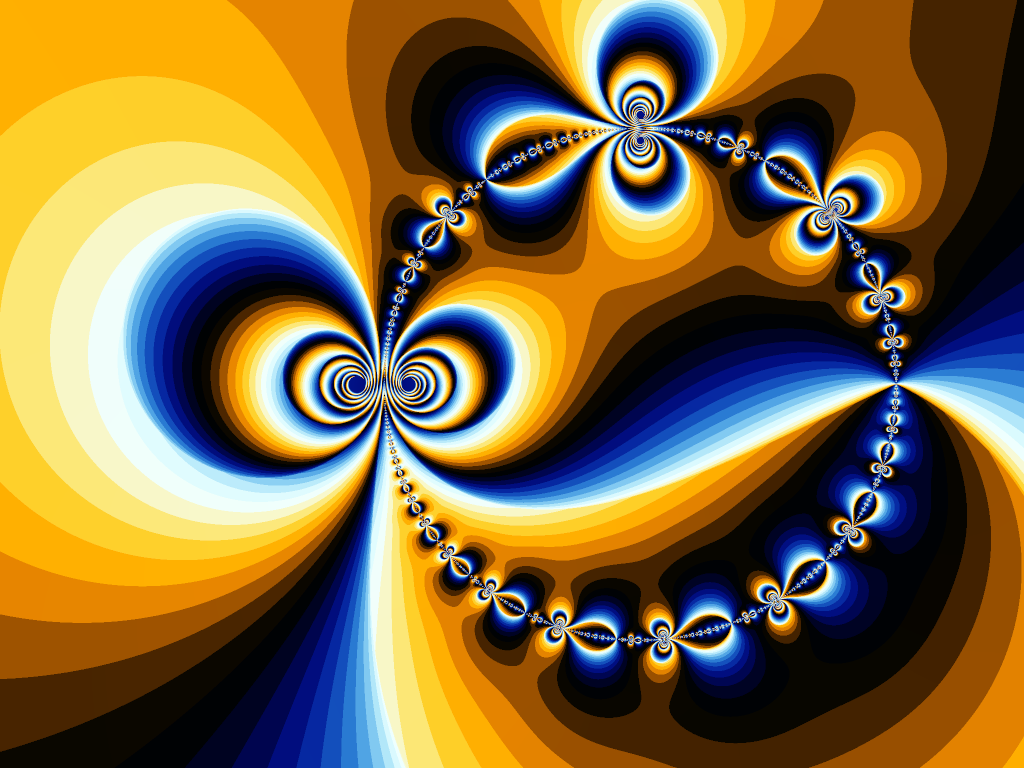}
\caption{The dynamics near $J(h_n)$}
\end{figure}

In order to illustrate the deformation of the dynamics of the twist sequence $\{f_n\}$ and its limit map near the Julia sets, we choose a map $[f_{\la_1,\la_2}]\in\HH$ such that $\la_1=\bar\la_2$. Then
$$
f_{\la_1,\la_2}(1/\bar z)\cdot\overline{f_{\la_1,\la_2}(z)}\equiv 1,
$$
i.e., $f_{\la_1,\la_2}$ is symmetric about the unit circle and $J(f_{\la_1,\la_2})$ is the unit circle.

Since $\la_1=\bar\la_2$, we may choose $\omega_1,\omega_2$ such that $e^{\omega_1}=\la_1$, $e^{\omega_2}=\la_2$ and $\omega_1=\bar\omega_2$. Let $\la_{1,n},\la_{2,n}$ be the eigenvalues of the attracting fixed points of the twist sequence $\{[f_n]\}$ of $f_{\la_1,\la_2}$ along $(\omega_1,\omega_2)$. Then $\la_{1,n}=\overline{\la_{2,n}}$ by the definition of twist deformation. Thus $f_{\la_{1,n},\la_{2,n}}$ is also symmetric about the unit circle.
By making a holomorphic conjugacy, we may choose a representative $h_n\in [f_n]$ such that $h_n$ is still symmetric about the unit circle, but the two attracting fixed points of $h_n$ are $-|\la_{1,n}|$ and $-1/|\la_{1,n}|$. Then as $n\to\infty$, $\{h_n\}$ uniformly converges to a quadratic rational map $h_{\infty}$. See Figure 2 for the dynamics of $h_n$ and Figure 3 for the dynamics of $h_{\infty}$ near Julia sets.

\begin{figure}[http]
\centering
\includegraphics[width=12cm]{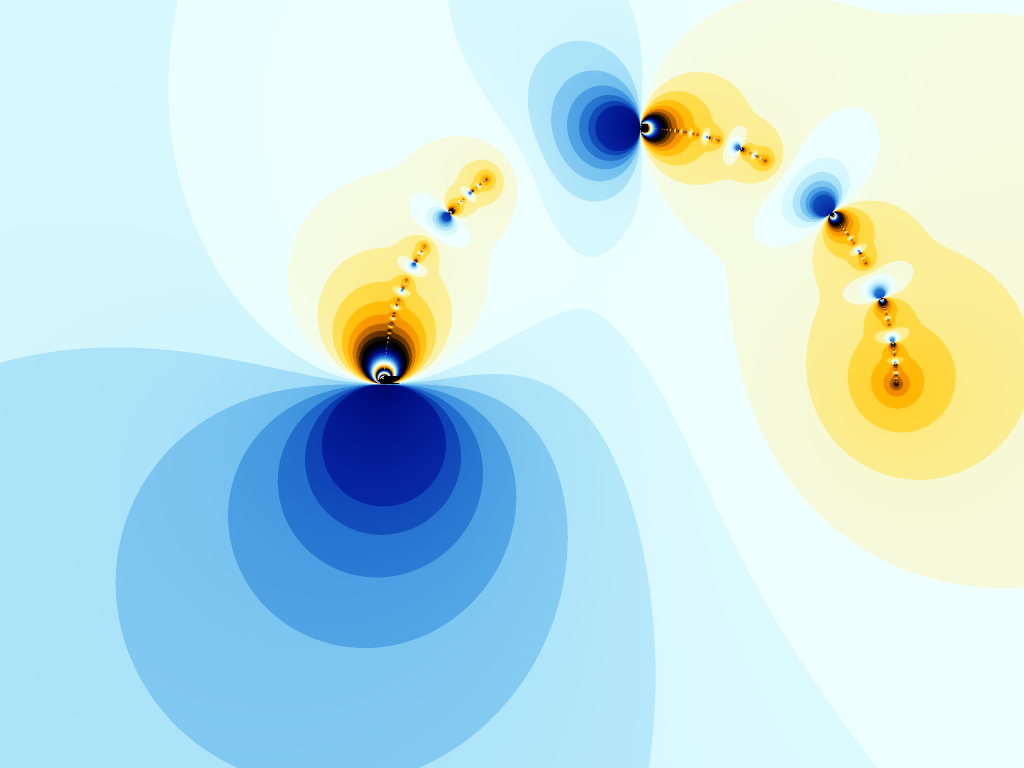}
\caption{The dynamics near $J(h_\infty)$}
\end{figure}

\section{The boundedness of $\MM_1$}
We may parameterize $\mathrm{Per}_1(1)$ by its eigenvalue at another fixed point.

\begin{lemma}\label{bound}
$\mathbf{M_1}\subset\{\la\in\C:\, |\la-1|\le 9\}$.
\end{lemma}

\begin{proof}
Each map in $\mathrm{Per}_1(1)$ can be represented by $g_B(z)=z+B+1/z$ with $B\in\C$. The infinity is a parabolic fixed point of $g_B$ with eigenvalue $1$. When $B=0$, $g_B(z)$ has no other else fixed point. Otherwise, $-1/B$ is the another fixed point of $g_B$ with eigenvalue $1-B^2$. Thus $g_B$ is holomorphically conjugate to $f_{1,\la}$ for $\la=1-B^2$.

Assume $|B|>3$. From
$$
\frac{g_B(z)}{B}-\frac{z}{B}=1+\frac{1}{Bz},
$$
we see that if $|z|>1$ and $\mathrm{Re}\,(z/B)>0$, then
$$
\mathrm{Re}\,\left(\frac{g_B(z)}{B}-\frac{z}{B}\right)=1+\mathrm{Re}\,\frac{1}{Bz}>\frac{2}{3}.
$$
Thus
$$
\left|\frac{g_B(z)}{B}\right|>\mathrm{Re}\,\frac{g_B(z)}{B}>\frac{2}{3}+\mathrm{Re}\,\frac{z}{B}>\frac{2}{3},
$$
and hence $|g_B(z)|>2|B|/3>1$.

Inductively, we obtain that if $|z|>1$ and $\mathrm{Re}\,(z/B)>0$, then $|g^n_B(z)|>1$ and $\mathrm{Re}\,(g_B^n(z)/B)>0$ for all $n\ge 1$. Consequently $\{g_B^n(z)\}\to\infty$ as $n\to\infty$.

The critical points of $g_B$ are located at $\pm 1$ and the corresponding critical values are $v_{\pm}=B\pm 2$. Thus when $|B|>3$, we have $|v_{\pm}|>1$ and
$$
\mathrm{Re}\,\frac{v_{\pm}}{B}=1\pm\mathrm{Re}\frac{2}{B}>0.
$$
Therefore the two critical points of $g_B$ converges to the infinity. So $J(g_B)$ is disconnected. Thus if $|\la-1|=|B^2|>9$, then $J(f_{1,\lambda})$ is disconnected. Now the lemma is proved.
\end{proof}

\begin{proof}[Proof of Theorem \ref{main}]
Combining Lemma \ref{12}, Theorem \ref{convergence} and Lemma \ref{dynamics}, we derive the statements (1) and (2) in Theorem \ref{main}.

From (1) and (2), we have $\BB_2\subset\partial\HH\cap(\mathrm{Per}_1(1)\smm\MM_1)\subset\overline{\BB}_2$. Lemma \ref{bound} shows that $\partial\HH\cap(\mathrm{Per}_1(1)\smm\MM_1)\neq\mathrm{Per}_1(1)\smm\MM_1$.
\end{proof}

\vspace{3mm}
\noindent{\it Acknowledgements}. The authors would like to express the sincere gratitude to the anonymous referee for all the valuable and helpful suggestions and comments.

\noindent
Guizhen Cui \\
School of Mathematical Sciences, \\
Shenzhen University, Shenzhen, 518052, P. R. China \\
and \\
Academy of Mathematics and Systems Science,\\
Chinese Academy of Sciences, Beijing, 100190, P. R. China.\\
gzcui@math.ac.cn

\vskip 0.24cm

\noindent
Wenjuan Peng \\
Academy of Mathematics and Systems Science, \\
Chinese Academy of Sciences, Beijing 100190, P. R. China.\\
wenjpeng@amss.ac.cn
\end{document}